\documentclass[12pt]{amsart}
 
\usepackage{amsfonts, amsmath, amssymb, amsgen, amsthm, amscd,
latexsym}
\usepackage[all]{xy}
\usepackage{color}

\numberwithin{equation}{section}

\newtheorem{thm}{Theorem}[section]

\newtheorem{theorem}{Theorem}[section]

\def\Diff{\mathop{\rm Diff}\nolimits}

\def\GL{\mathop{\rm GL}\nolimits}
\def\SO{\mathop{\rm SO}\nolimits}

\def\Id{\mathop{\rm Id}\nolimits}

\def\det{\mathop{\rm det}\nolimits}

\def\log{\mathop{\rm log}\nolimits}

\def\Tr{\mathop{\rm Tr}\nolimits}

\def\Cb{{\mathbb C}}

\def\Nb{{\mathbb N}}
\def\Rb{{\mathbb R}}

\def\Zb{{\mathbb Z}}

\def\Fc{{\mathcal F}}

\def\Hc{{\mathcal H}}

\def\Ic{{\mathcal I}}

\def\Uc{{\mathcal U}}

\def\Qc{{\mathcal Q}}

\def\a{\alpha}

\def\d{\delta}
\def\e{\epsilon}
\def\D{\Delta}
\def\g{\gamma}

\def\G{\Gamma}

\def\lb{\lambda}

\def\om{\omega}
\def\Om{\Omega}
\def\s{\sigma}

\def\ve{\varepsilon}
\def\vp{\varphi}

\def\Gb{{\bf G}}

\def\fl{\forall}
\def\ify{\infty}

\def\ot{\otimes}
\def\bot{\bar\otimes}

\def\ra{\rightarrow}

\def\p{\partial}

\def\0D{\Delta^{(0)}}
\def\1D{\Delta^{(1)}}

\newcommand{\Fa}{\mathfrak{a}}
\newcommand{\Fg}{\mathfrak{g}}
\newcommand{\Fh}{\mathfrak{h}}

\newcommand{\Fl}{\mathfrak{l}}

\def\Gb{{\bf G}}

\def\0b{{\bf 0}}
\def\1b{{\bf 1}}

\def\build#1_#2^#3{\mathrel{
\mathop{\kern 0pt#1}\limits_{#2}^{#3}}}

\def\one{{\bf 1}}

\def\a{\alpha}

\def\d{\delta}
\def\e{\epsilon}
\def\g{\gamma}

\def\lb{\lambda}
\def\om{\omega}
\def\s{\sigma}

\def\ve{\varepsilon}

\def\vp{\varphi}

\def\D{\Delta}
\def\G{\Gamma}

\def\Om{\Omega}

\def\fl{\forall}
\def\ify{\infty}

\def\ot{\otimes}
\def\part{\partial}

\newcommand{\ie}{{\it i.e.\/}\ }

\newcommand{\cf}{{\it cf.\/}\ }

\def\ra{\rightarrow}

\def\text{\hbox}

\def\fl{\forall \,}
\def\ify{\infty}

\def\ot{\otimes}

\def\ra{\rightarrow}

\def\Diff{\mathop{\rm Diff}\nolimits}

\def\GL{\mathop{\rm GL}\nolimits}
\def\SO{\mathop{\rm SO}\nolimits}

\def\Id{\mathop{\rm Id}\nolimits}

\def\cutint{{\int \!\!\!\!\!\! -}}

\def\build#1_#2^#3{\mathrel{
\mathop{\kern 0pt#1}\limits_{#2}^{#3}}}

\numberwithin{equation}{section}

\title{On the van Est analogy in Hopf cyclic cohomology}

\begin{document}
 
\author{Henri Moscovici}
 \address{Department of mathematics,  
The Ohio State University, 
Columbus, OH 43210, USA}
\email{moscovici.1@osu.edu}



 \maketitle


\section*{Introduction} \label{intro}

Hopf cyclic cohomology has emerged from the attempt  
to unravel the local formula~\cite{CM95} for the cyclic cohomological Chern character ~\cite{ncdg}
of the hypoelliptic spectral triple~\cite{CM98} modeling the space of leaves of a foliation. 
A close inspection of that formula brought to light a Hopf algebra of `moving frames'
$\Hc(n, \Rb)$ (abbreviated as $\Hc_n$), which appeared to assume
in the transverse geometry of foliations a role similar to that of
$\GL(n, \Rb)$ as structure group of the frame bundle of a manifold.
Further examination revealed that the the index formula in question 
was in fact the expression of a characteristic 
cocycle in the range of a certain cohomology specific to Hopf algebras. 
In the particular case of $\Hc_n$, this Hopf cyclic cohomology 
turned out to be isomorphic to the Gelfand-Fuchs cohomology of the
Lie algebra $\Fa_n$ of formal vector fields on $\Rb^n$ which, together with its
relative versions, is well-known to encode the geometric characteristic classes of foliations. 
The benefit of having the same classes recaptured
in Hopf cyclic cohomology is that the latter affords a direct characteristic map to the cyclic 
cohomology of \'etale foliations groupoids. 
Ultimately, the main outcome of the 
aforementioned investigation was the realization that the Hopf cyclic cohomology
$HP^\bullet(\Hc_n; \Cb_\d)$ of $\Hc(n, \Rb)$ together with its relative versions 
$HP^\bullet(\Hc_n, {\rm O}_n; \Cb_\d)$ and $HP^\bullet(\Hc_n, \GL(n, \Rb); \Cb_\d)$ provide 
the appropriate repository for all geometric
characteristic classes of foliations in the framework of noncommutative geometry. 

On the other hand in the classical theory it was known that the characteristic classes 
can be constructed not just in terms of the usual connection and curvature procedure, 
but also by simplicial cohomological methods involving $\GL(n, \Rb)$ as structure group
(\cf~\cite{Bott-ChWe,  BSS, KT, ShSt, Dupont}).
A. Gorokhovsky figured out how to convert that kind of representation
into the formalism of Hopf cyclic cohomology by extending it to differential graded 
(DG) Hopf algebras. He showed~\cite{Gor}  
that the suitably truncated Hopf cyclic cohomology $HP^\bullet((\Om^*\GL(n, \Rb))_m)$
of the DG Hopf algebra of differential forms on $\GL(n, \Rb)$ captures
the characteristic classes of flat $\G$-equivariant vector bundles of rank $n$, where $\G$ is a discrete group of diffeomorphisms of the base manifold of dimension $m$. 
In particular his results yield an isomorphism
\begin{align}   \label{G}
HP^\bullet(\Hc_n, {\rm O}_n; \Cb_\d) \xrightarrow{\cong} HP^\bullet((\Om^*\GL(n, \Rb))_n) ,
\end{align}
which bears a conspicuous resemblance with the van Est isomorphism for the continuous cohomology 
of $\GL(n, \Rb)$. 

The aim of this note is to show that this analogy is not coincidental, and it in fact extends to 
other avatars of the van Est isomorphism, thus providing the connection with the absolute cohomology 
$HP^\bullet(\Hc_n; \Cb_\d)$, which accounts for the characteristic classes of
framed foliations, and with the relative cohomology $HP^\bullet(\Hc_n, \GL(n, \Rb); \Cb_\d)$,
which only  stores the Chern classes. More specifically, we prove that in addition to
 \eqref{G} there are canonical isomorphisms
\begin{align} \label{Galg}
HP^\bullet(\Hc_n, \GL(n, \Rb); \Cb_\d) \xrightarrow{\cong} 
HP^\bullet((\Om^*\GL_{\rm alg}(n, \Rb))_n) ,
\end{align}
where $\Om^*\GL_{\rm alg}(n, \Rb)$ is the Hopf algebra of forms on $\GL(n, \Rb)$
viewed as an algebraic group, and 
\begin{align} \label{Gabs}
HP^\bullet(\Hc_n; \Cb_\d) \xrightarrow{\cong} 
HP^\bullet((\bar\Om^*\GL_{\rm germ}(n, \Rb))_n) ,
\end{align}
where $\bar\Om^*\GL_{\rm germ}(n, \Rb)$ is the Hopf algebra of germs of forms on the 
group germ $\GL_{\rm germ}(n, \Rb)$ defined by $\GL(n, \Rb)$.

The material is organized as follows. After recalling the basic definitions related to $\Hc_n$
and Hopf cyclic cohomology in \S \ref{Hn}, we outline 
Gorokhovsky's computation of the cyclic cohomology of $\Om^*\GL(n, \Rb)$ in \S \ref{DGn},
explain the `behind-the-scenes' role of the van Est isomorphism in \S \ref{diff}, and then
prove the above mentioned results in \S \ref{salg} (Theorem \ref{alg}) and 
 \S \ref{sgerm} (Theorem \ref{germ}). 
Finally, in \S \ref{fin} we illustrate the convenience of using $HP^\bullet((\Om^*\GL(n, \Rb))_n)$ as 
a repository of characteristic classes by giving examples of representative cocycles.

\section{Hopf algebra of moving frames and its cyclic cohomology} \label{Hn}

We recall in this section the definition of the Hopf algebra of moving frames in $\Rb^n$,
its Hopf cyclic cohomology, and how the relationship with the Gelfand-Fuchs
cohomology of the Lie algebra of formal vector fields on $\Rb^n$.
  
\subsection{Hopf algebra $\Hc_n$}
Consider the groupoid $\G_n= \Diff_n \ltimes \Rb^n$, where $ \Diff_n$ denotes the group
of diffeomorphisms of $\Rb^n$ equipped with the discrete topology. 
Let  $F\G_n= \Diff_n \ltimes F\Rb^n$, where $F\Rb^n= \Rb^n  \times \GL(n, \Rb)$
denote the `moving frames' groupoid, with $ \Diff_n$ acting by prolongation:
 \begin{equation*} 
{\vp} (x, {\bf y}) := \left( \vp (x), {\vp}^{\prime} (x)  
{\bf y} \right), \,  \text{where} \, \, \vp^{\prime}(x)
 = \left(\part_{j} \vp^{i}  (x)\right) \in  \GL(n, {\Rb}).
\end{equation*}
The convolution algebra $C_c^\ify (F\G_n)$ is spanned by monomials 
$f U_\vp$ with $f \in C_c^\ify (F\Rb^n)$ and $\vp \in \Diff_n$, with the product given by the rule
\begin{align*}
f_1 U_{\vp_1}\,\cdot  f_2 U_{\vp_2} = f_1 \left(f_2 \circ {\vp_1}^{-1}\right) \, U_{\vp_1 \vp_2} ,
\quad \text{where} \quad U_\vp(f)  =  f \circ \vp^{-1}.
\end{align*}
The vertical vector fields
  $Y_i^j = \sum_\mu y_i^{\mu} \frac{\p}{\p\, y_j^{\mu}}$ and the horizontal vector fields
$X_k = \sum_\mu y_k^{\mu} \frac{\p} {\p \, x^{\mu}}$, are made to act on 
$C_c^\infty(F\G_n)$ as follows: if $Z$ is one of these, then
  $Z (f U_{\vp})  :=  Z (f )U_{\vp}$. 
Since the right action of $\GL(n,\Rb)$ on  $F{\Rb}^n$ commutes with
the action of $\Diff_n$, the vertical operators so extended remain derivations:
\begin{equation*}  
Y_i^j (a \, b) \, = \, Y_i^j (a) \, b \, + \, a \,Y_i^j (b) \,   ,
\quad a, b \in C_c^\infty(F\G_n) \, .
\end{equation*}
By contrast, the horizontal vector fields are not $\Diff_n$-invariant, and satisfy instead
\begin{align*} 
\begin{split}
 \vp_*(X_k) \equiv U_\vp \, X_k \, U_{\vp}^{-1} &  =  \,  X_{k} \,  - \, \g_{jk}^i (\vp^{-1}) \, Y_i^j  \, , \qquad
\text{where} \\
\g_{jk}^i (\vp) (x, {\bf y}) & =\, \left( {\bf y}^{-1} \cdot
{\vp}^{\prime} (x)^{-1} \cdot \part_{\mu} {\vp}^{\prime} (x) \cdot
{\bf y}\right)^i_j \, {\bf y}^{\mu}_k \, .
\end{split}
\end{align*}
Consequently the extended operators $X_k$ are no longer derivations but satisfy
the generalized Leibniz rule
\begin{equation*}  
X_k(a \, b) \, = \, X_k (a) \, b \, + \, a \,X_k (b) \, + \,
\d_{jk}^i (a) \, Y_{i}^{j} (b)
 \,   , \quad
a, b \in C_c^\infty(F\G_n) .
\end{equation*}
Here $\, \d_{jk}^i $ are multiplication operators given by
\begin{equation*}   
\d_{jk}^i (f  U_{\vp})  =\, \g_{jk}^i  (\vp^{-1}) \, f  U_{\vp}  =\, 
- \g_{jk}^i (\vp) \circ \bar\vp^{-1} \, f  U_{\vp}  
\end{equation*}
and they satisfy the usual Leibniz rule
\begin{equation*} 
\d_{jk}^i (a \, b) \, = \, \d_{jk}^i  (a) \, b \, + \, a \,
\d_{jk}^i  (b) .
\end{equation*}
The operators  $\Id, \, \{X_k , \, Y^i_j  , \, \d_{jk}^i \} \,$ 
generate a unital subalgebra  of the linear operators on $C_c^\ify (F\G_n)$, 
which defines $\Hc_n$ as an algebra. Automatically,
$\Hc_n$ is also a Lie algebra with respect to the usual commutator, and
the commutator relations between its generators give more insight into its nature.
The vector fields satisfy the standard commutation
relations of the affine group, and the
iterated commutators of the operators $\, \d_{jk}^i $'s with
the $X_{\ell}$'s yield further operators
\begin{equation*} 
\d_{jk \,  \ell_1 \ldots \ell_r}^i \, := \, [X_{\ell_r} , \ldots
[X_{\ell_1} , \d_{jk}^i] \ldots] ,
\end{equation*}
which act by multiplication
\begin{eqnarray*}  
\d_{jk \,  \ell_1 \ldots \ell_r}^i \, ( f \, U_{\vp}) \, &:=& \,
 \g_{jk \,  \ell_1 \ldots \ell_r}^i (\vp^{-1})\, f \, U_{\vp} \, , \qquad \text{where}  \\
  \g_{jk \,  \ell_1 \ldots \ell_r}^i (\vp)\, &:=&
\, X_{\ell_r} \cdots X_{\ell_1} \big(\g_{jk}^i (\vp)\big) ,  \qquad \vp \in \bar\Gb_n ,
\end{eqnarray*}
and therefore commute with each other. They are no longer derivations but do satisfy
(generalized) Leibniz rules.

The Leibniz rules 
obeyed by the generators extend by multiplicativity to all $h \in \Hc_n$. In turn,
these rules uniquely determine a coproduct $\D:\Hc_n \to \Hc_n \ot \Hc_n$ 
by the stipulation expressed in the Sweedler notation as follows:
\begin{align} \label{cop}
 \D(h) = \sum h_{(1)} \ot h_{(2)}  \quad \text{iff} \quad h(a b)= \sum  h_{(1)} (a) h_{(2)} (b) .
\end{align}
The counit  $\e : \Hc_n \to \Cb$, $\e (h):= h(1)$ 
and the anti-automorphism $S: \Hc_n \to \Hc_n$ defined on generators by
\begin{align*}
S(Y_i^j) = - Y_i^j , \quad S(X_k) = -X_k + \d_{jk}^i Y_i^j , \quad S(\d_{jk}^i ) = -  \d_{jk}^i ,
\end{align*} 
complete the properties which make $\Hc_n$ a Hopf algebra. 

The antipode $S$ is not involutive but it can be twisted into an involution by means of 
the character $\d : \Hc_n \to \Cb$ which trivially extends the trace map on $\Fg \Fl_n (\Cb)$, 
\ie defined by
\begin{align*}
\d(Y_i^j) = \d_i^j , \quad \d(X_k) =0 , \quad \text{and} \quad \d(\d_{jk}^i ) =0 .
\end{align*} 
 Then $S_\d (h) := \sum \d(h_{(1)}) S (h_{(2)} )$, $h \in \Hc_n$ satisfies $S_\d^2 = \Id$.

\subsection{Hopf cyclic cohomology}  
By its very construction, the Hopf algebra $\Hc_n$ acts on the algebra $C_c^\infty(F\G_n)$,
and \eqref{cop} ensures that $C_c^\infty(F\G_n)$ is actually a left $\Hc_n$-module algebra. 
$C_c^\infty(F\G_n)$ has a canonical trace, namely
\begin{align}  \label{tau}
 \tau (f U_\vp) = \int_{F{\Rb}^n} f \varpi , \quad  \text{if} \,\, \vp=\Id , \qquad 
  \tau (f U_\vp) = 0 , \quad  \text{if} \,\, \vp \neq \Id , 
 \end{align} 
 where $\varpi$ is the volume form dual to  
 $X_1 \wedge \cdots \wedge X_n \wedge Y^1_1 \wedge \cdots \wedge 
Y^n_n $ (lexicographically ordered); $\tau$
is tracial since $\varpi$ is $\Diff_n$-invariant, and
in addition it is $\d$-invariant with respect to the action of $\Hc_n$, \ie satisfies
\begin{align*}  
\tau (h(a)) \, = \,  \d(h)\, \tau(a) , \qquad \fl \, h\in\Hc , \, \, a \in C_c^\infty(F\G_n) .
\end{align*} 

The cyclic cohomology of the Hopf algebra $\Hc_n$ was introduced in~\cite{CM98} 
by importing the standard cyclic structure of the algebra $C_c^\infty(F\G_n)$.
via the morphism $\chi_* = \{\chi_q: \Hc_n^{\ot \, q} \ra CC^q \left(C_c^\ify (F\G_n) \right)\}_{q \geq 0}$, where
\begin{align} \label{chi} 
 \chi_q (h^1 \ot \ldots \ot h^{q})(a_0, a_1, \ldots, a_q) := 
\tau \left(a_0 h^1(a_1) \cdots h^q(a_q)\right) . 
 \end{align}
The stipulation that $\chi_*$ is a morphism of $\D C$-objects, where $\D C$ is Connes' cyclic
category, endows
$\Hc_n^\natural := \sum_{q\geq 0}^\oplus \Hc_n^{\ot q}$ with the basic
$\D C$-morphisms
 \begin{align} \label{co}
 \begin{split}
\p_0 (h^1 \ot \ldots \ot h^{q-1}) &= 1 \ot h^1 
\ot \ldots \ot h^{q-1}, \\
\p_j (h^1 \ot \ldots \ot h^{q-1}) &= h^1 \ot \ldots \ot \D h^j \ot 
\ldots \ot h^{q-1}  \\
\p_q (h^1 \ot \ldots \ot h^{q-1}) &= h^1 \ot \ldots \ot h^{q-1}
\ot 1 \\
\s_i (h^1 \ot \ldots \ot h^{q+1})  &=  h^1 \ot \ldots \ot \ve (h^{i+1}) 
\ot \ldots \ot h^{q+1} \\
\tau_q (h^1 \ot \ldots \ot h^q)&= S_\d (h^1) \cdot (h^2 
\ot \ldots \ot h^q \ot 1).
\end{split}
\end{align}
The involutive property $ S^2_\d = \Id $ evidently implies $\tau_1^{2} = \Id$, and
in fact ensures that $\tau_q^{q+1} = \Id$ for any $q\in \Nb$ (\cf~\cite{CM99}).
 The corresponding cyclic cohomology groups $HC^*(\Hc;  \Cb_\d)$
are computed from the associated bicomplex $CC^{*,*} (\Hc_n;  \Cb_\d)$  with boundary operators
 \begin{align} \label{cc}
 b = \sum_{i=0}^{q+1} (-1)^i \p_i , \quad
 B = \bigl( \sum_{i=0}^q (-1)^{(q-1)i}\tau_q^i \bigr) \, \sigma_{q-1} \, (1- (-1)^q\tau_q) ,
\end{align}
and the periodic Hopf cyclic cohomology $HP^\bullet (\Hc_n; \Cb_\d)$ is 
 the $\Zb_2$-graded cohomology of the corresponding total complex.
 
More generally (\cf~\cite[Theorem 1]{CM99}), to any Hopf algebra $\Hc$ endowed with
a {\em modular pair in involution} $(\s, \d)$, \ie $\s$ is a group element, $\d$ a character
and $\d(\s)=1$, one associates a $(b, B)$ bicomplex in a similar way, 
after modifying the operators in \eqref{co} as follows: 
 \begin{align} \label{Hco}
 \begin{split}
\p_q (h^1 \ot \ldots \ot h^{q-1}) &= h^1 \ot \ldots \ot h^{q-1}
\ot \s \\
\tau_q (h^1 \ot \ldots \ot h^q)&= S_\d (h^1) \cdot (h^2 
\ot \ldots \ot h^q \ot \s).
\end{split}
\end{align}
The definition of Hopf cyclic cohomology was subsequently extended 
in~\cite{hkrs} to a large class of coefficients.

For applications to characteristic classes of foliations, the relative versions of the
Hopf cyclic cohomology of $\Hc_n$ of particular interest are that
relative to  ${\rm O}_n$, resp. to $\SO_n$ in the orientable case,
and to a lesser extent that relative to $\GL(n, \Rb)$. 

To simplify the notation,  in the remainder of this section we abbreviate $\GL(n, \Rb)$
by  $\GL_n$ and $\Fg \Fl_n (n, \Rb)$ by $\Fg_n$. 
Letting $H$ denote one of closed subgroups of $\GL_n$ enumerated above, we recall that
the cohomology of $\Hc_n$ relative to $H$ is defined as follows (\cf \cite{CM-big}). 

First note that the adjoint action 
of $\GL_n$ on the Lie algebra $\Fg\Fl_n \ltimes \Rb^n$ of the affine group
$\GL_n \ltimes \Rb^n$, with which we identify the frame bundle $F\Rb^n$, 
extends to a linear action of  $\GL_n$ on $\Hc_n$. 
With $\Fh$ denoting the Lie algebra of $H$,
its universal enveloping algebra $\Uc (\Fh)$ is a sub-Hopf algebra of $\Hc_n$.
The quotient $\Hc_n \ot_{\Uc (\Fh)} \Cb \equiv \Hc_n /\Hc_n \Uc^+ (\Fh)$ (where
$\Uc^+ (\Fh)$ denotes the ideal of $\Uc(\Fh)$ generated by $\Fh$)
  is an $\Hc_n$-module coalgebra with respect to the coproduct and counit
 induced from $\Hc_n$. We denote by $\Qc_H$ its subspace of invariants under the  
 action of $H$ on $\Hc_n$.
The cyclic object defining the relative cohomology consists of 
$ \sum_{q\geq 0}^\oplus \Qc_H^{\ot q}$ 
endowed with the basic $\D C$-morphisms defined by 
 \begin{eqnarray} \label{relco}
 \begin{split} 
 \p_0 (c^1 \ot \ldots \ot c^{q-1}) &=& 
  \dot{1} \ot c^1 \ot \ldots \ot
  \ldots \ot c^{q-1} ,   \\  
\p_i  (c^1 \ot \ldots \ot c^{q-1})  &=& 
c^1 \ot \ldots \ot \D c^i \ot \ldots \ot c^{q-1} ,  \\ 
\p_n  (c^1 \ot \ldots \ot c^{q-1})  &=& c^1 \ot \ldots \ot c^{q-1}
\ot \dot{1}\, ; \\  
 \s_i  (c^1 \ot \ldots \ot c^{q+1})  &=& c^1 \ot \ldots \ot \ve
(c^{i+1}) \ot \ldots \ot c^{q+1} , \\   
   \tau_q ( \dot{h}^1 \ot c^2 \ot \ldots \ot c^q) &=& 
  {S_\d}(h^1) \cdot (c^2
   \ot \ldots \ot c^q \ot \dot{1}) .
 \end{split}
\end{eqnarray}
The corresponding bicomplex is denoted $CC^* (\Hc_n, H;  \Cb_\d)$ 
and the $\Zb_2$-graded 
cohomology group of the total complex is $HP^\bullet (\Hc_n, H ; \Cb_\d)$.

The above cohomology groups have been computed in \cite{CM95} (and by alternative means  
in \cite{MR09, MR11}). They are isomorphic to the Gelfand-Fuchs 
cohomology of the Lie algebra $ \Fa_n$ of formal vector fields
on $\Rb^n$, resp. to its relative versions, which in turn are isomorphic to
the cohomology of truncated Weil algebras (\cf \cite{GF, Godbillon}), as follows: 
 \begin{align} \label{Habs}
 HP^\bullet(\Hc_n; \Cb_\d) \cong H_{\rm cont}^{\bullet}(\Fa_n; \Cb) 
\cong  H^\bullet({W}(\Fg \Fl_n)_n)  , 
\end{align}
and relative to $H={\rm O_n}, \SO_n$ or  $\GL(n, \Rb)$ 
  \begin{align} \label{Hrel}
 HP^\bullet(\Hc_n, H; \Cb_\d) \cong 
H_{\rm cont}^{\bullet}\bigl(\Fa_n,H; \Cb\bigr) \cong  H^\bullet\bigl(W(\Fg \Fl_n, H)_n\bigr) .
\end{align}
Here  $W(\Fg \Fl_n)= \wedge \Fg \Fl_n^* \otimes S(\Fg \Fl_n^*)$ is the Weil algebra,   
${W}(\Fg \Fl_n)_m$ denotes the
 quotient by the ideal generated by $\sum_{k >m}S^k(\Fg \Fl_n^*)$, and
 ${W}(\Fg \Fl_n, H)_m$ stands for the DG subalgebra of $H$-basic elements in ${W}(\Fg \Fl_n)_m$. Explicit bases for the cohomology groups in the right hand side of
 \eqref{Habs} and \eqref{Hrel} are given in \cite{Godbillon}, and their geometric realizations
via the above isomorphisms can be found in \cite{Mos2015, Mos2015b}.

  \section{Hopf algebras of forms and their cyclic cohomology} \label{DG}
  
 \subsection{DG Hopf algebra of a Lie group} \label{DGn}
 Let $G$ be an almost connected Lie group. Its algebra of differential forms $\Om^*(G)$ 
has an natural structure of a topological (Fr\'echet) DG Hopf algebra, extending that of the function
algebra $C^\ify(G)=\Om^0(G)$. Indeed $ C^\ify(G) $,
equipped with its standard topological vector space structure 
of uniform convergence on compacta of functions and their 
derivatives, is a Hopf algebra with
the coproduct $\D: C^\ify(G) \to C^\infty(G) \hat{\otimes}C^\infty(G)=C^\infty(G \times G)$ 
defined by $\D(f)(x,y)=f(xy)$, the counit given by evaluation at 
the unit $\one \in G$ and the antipode $S$ induced by inversion. As topological vector spaces
$\Om^*(G) \cong C^\ify(G) \ot \bigwedge \Fg^*$, where $\Fg$ is the Lie algebra of $G$, and the  Hopf structure of $C^\ify(G) = \Om^0(G)$ extends naturally to $\Om^*(G)$. 
Together with the de Rham differential, 
the couple  $\left(\Om^*(G), d \right)$ is a DG (Fr\'echet) Hopf algebra. 

In~\cite[\S 3]{Gor} Gorokhovsky defines a cyclic object associated to an arbitrary  DG Hopf algebra
$(\Hc^*, d)$ endowed with a modular pair, by a natural graded extension of the cyclic object 
associated to $\Hc^0$. In particular this applies to $\left(\Om^*(G), d \right)$ which, 
having involutive antipode, carries a trivial modular pair. The cyclic object
consists of 
$$(\Om^*G)^\natural := \sum_{q\geq 0} \Om^*(G)^{\hat {\otimes}q} \equiv 
\sum_{q\geq 0} \Om^*(G^{\times q}) 
$$
equipped
with the $\D C$-morphisms similar to those in \eqref{co}, with the difference that the 
summands (implicitly present) in the expression 
of the cyclic operator $\tau_q$ acquire appropriate signs (\cf \cite[\S 3,  (3.3)]{Gor}).  
The corresponding $(b, B)$-bicomplex $C^{*,*}(\Om^*(\GL(n, \Rb)))$ 
is then upgraded to a tricomplex $C^{*,*,*}(\Om^*(G),d)$ by the addition of the differential 
\begin{align*}   
{d}(\a_1 \otimes \cdots \otimes \a_q)= \sum_{i=1}^q (-1)^{\deg \a_1 + \cdots +\deg\a_q}
\a_1 \otimes \cdots d\a_i \cdots \otimes \a_q) ,
\end{align*}
The latter is filtered by the subcomplexes
\begin{align*}  
  F^mC^\bullet(\Om^*(G), d)=
 \{\sum \a_1 \otimes \cdots \otimes \a_q \mid \, \deg \a_1 + \cdots +\deg\a_q > m \} ,
\end{align*}
and gives rise, for each $m \in \Zb^+$, to a truncated complex
\begin{align*}  
  C^{*,*,*}(\Om^*(G)_m,d) =
 C^{*,*,*}(\Om^*(G), d)/F^mC^\bullet(\Om^*(G), d) ;
\end{align*}
Its Hochschild, cyclic and periodic cyclic cohomology, formed using only finite cochains,
 will be respectively denoted by
 $HH^*\left((\Om^*G)_m\right)$, $HC^*\left((\Om^*G)_m\right)$ and 
 $HP^\bullet\left((\Om^*G)_m\right)$. 
Gorokhovsky computed them in terms of the cohomology of truncated Weil algebras 
as follows.

\begin{theorem}[\cite{Gor}, \S 6] \label{Gor-Iso}
With $K$ denoting the maximal subgroup of $G$, there are canonical isomorphisms 
\begin{align} \label{muHH} 
HH^*\left((\Om^*G)_m\right) & \cong  H^*\bigl(W(\Fg, K)_m\bigr), 
 \quad \fl \, m \geq 0,\\  \label{muHC}
 HC^q\left((\Om^*G)_m\right) &\cong \bigoplus_{i \geq 0}H^{q-2i}\bigl(W(\Fg, K)_m\bigr).
\end{align}
\end{theorem}
 
For the clarity of the ensuing discussion we outline Gorokhovsky's proof.  
The crucial observation is that the Hochschild 
bicomplex $C^{*,*}(\Om^*(G), b, d)$ coincides with the Bott~\cite{Bott-ChWe}
simplicial de Rham bicomplex $ \{\Om^{*,*}(NG), \d, d\}$ of the nerve $NG$ of
$G$. Recall that $NG:=\bigsqcup N_pG$, where $N_pG=G^p$, and $\Om^{p,q}(NG) =
 \Om^q(G^p)$.
 
As was shown in \cite{BSS} for a general simplicial manifold $X=\{X_p\}$, the cohomology
of $ \{\Om^{*,*}(X), \d, d\}$ is isomorphic to the singular cohomology of the 
geometric realization of $X$. On the other hand,
Dupont \cite{Dupont}) introduced a related de Rham complex, $ \{\Om^*(||X||), d\}$,
where $||X||$  is the ``thick'' geometric realization of $X$,
which has the advantage of being graded commutative. A $q$-form $\om\in \Om^q(||X||)$ is a 
collection $\{\om_p\}$ of de Rham $q$-forms on $\bigsqcup_p X_p$,
satisfying the compatibility condition
$$
(\varepsilon^i \times \Id)^*\om_p = (\Id \times \varepsilon_i)^*\om_{p-1} ,
$$
where $\varepsilon_i:X_p \to X_{p-1}$ are the face operators corresponding to the
inclusions of the faces $\varepsilon^i:\D^{p-1}\to \D^p$, $i=0,1, \ldots , p$. 
Dupont has shown that integration over simplices $\Ic_\D: \Om^*(||X||) \to \Om^*(X)$,
$$ \Ic_\D(\a) = \int_{\D^p} \a \mid \D^p \times X_p 
$$
 defines a quasi-isomorphism of $ \{\Om^*(||X||), d\}$ with 
with the total complex $ \{\Om^{*}(X), \d \pm d\}$.

We next recall that
 the geometric realization $EG$ of the simplicial manifold $\bar{N}G$, with 
$\bar{N}_p G = G^{p+1}$ on which $G$ acts diagonally, gives 
the total space of the universal (right) $G$-bundle $\pi: EG \to BG$.
This bundle has a canonical connection $\theta$ induced by the
Maurer-Cartan form on $G$; for a linear group $G$ the expression of
$\theta \mid \D^p \times \bar{N}_p G$ is
\begin{align} \label{conn}
\theta (t_0, \ldots, t_p; g_o, \ldots, g_p)= t_0 g_o^{-1}dg_o + \cdots +t_p g_p^{-1}dg_p ,
\end{align}
and the curvature form is $\Om= d\theta + \theta\wedge\theta$.

The Dupont complex being a $\Fg$-DG algebra, by the universal property of the Weil algebra 
(see \cite{Cartan}) this connection determines a morphism  $w: W(\Fg) \to \Om^*(EG)$.
In turn, the latter descends to $K$-invariants yielding a morphism 
 $w_K: W(\Fg, K) \to \Om^*(EG/K)$. Using the contractibility of $G/K$ and the
 canonical construction of geodesic simplices (\cf \cite[\S 5]{Dupont} for
 metrics of nonpositive curvature, one defines a simplicial
 cross-section $s: BG \to EG/K$ (\cf \cite[Eqs. (6.10),(6.11)]{Gor}), which moreover
 is equivariant with respect
 to the cyclic group action (\cf \cite[ (6.8)]{Gor}). The composition
 \begin{align} \label{mu}
\mu = \Ic_\D\circ \, s^*\circ w_K: W(\Fg, K) \to \Om^*(NG)  
\end{align}   
preserves the canonical filtrations and descends to the truncated complexes of 
any level $m \geq 0$, 
inducing in cohomology a morphism
\begin{align} \label{WH}
\mu^*_{HH} : H^*\left(W(\Fg, K)_m\right) \to H^q\left((\Om^*(NG))_m\right) =
HH^q\left((\Om^*G)_m\right) .
\end{align}
In view of results proved in~\cite{KT} and \cite{ShSt}, it follows that $\mu^*_{HH}$ is actually 
an isomorphism. 
Finally, since the connection \eqref{conn} and the cross-section $s$ are invariant 
under the natural cyclic action, the Hochschild cocycles in the image of $\mu$ are actually
cyclic. This implies that
Connes' periodicity exact sequence splits into short exact sequences
 \begin{align*} 
0 \to HC^{q-2}\left((\Om^*G)_m\right) \xrightarrow{S} HC^q\left((\Om^*G)_m\right)
\xrightarrow{I} HH^q\left((\Om^*G)_m\right) \to 0 ,
\end{align*}   
which in turn implies the isomorphism \eqref{muHC}
for cyclic cohomology.

\subsection{Role of van Est isomorphism} \label{diff}
 We now revisit the proof of the isomorphism \eqref{WH} and give an alternative 
argument emphasizing the role of the  van Est isomorphism, which will serve
as a template for dealing with the counterparts of the isomorphism \eqref{muHC} 
corresponding to $H=\{\1b\}$ in \eqref{Habs} and and $H=\GL(n, \Rb)$ in  \eqref{Hrel}.

A key element of the approach we are about to describe is the linkage
made by Bott~\cite{Bott-ChWe} with the continuous group cohomology. 
Relying on the Dold-Puppe homology theory for nonadditive functors,
Bott proved a basic Decomposition Lemma~\cite[\S 2]{Bott-ChWe} for the 
simplicial de Rham complex $\Om^*(NG)$ and use it show that
for any Lie group $G$ the spectral sequence 
of the bicomplex $\{\Om^*(NG), d, \d\}$ filtered by the degree of the forms
converges to the cohomology of $BG$. In particular (\cf~\cite[Theorem 1]{Bott-ChWe}) 
he expressed the $E_1$-term as continuous group cohomology with coefficients:
\begin{align} \label{BvE}
E_1^{pq}= H_\d^p \bigl(\Om^q(NG)\bigr) \cong H_{\rm cont}^{p-q}\bigl(G; S^q(\Fg^*)\bigr) .
\end{align}   
By the van Est isomorphism~\cite{vEst},  
\begin{align} \label{vE} 
 H_{\rm cont}^*\bigl(G; S^*(\Fg^*)\bigr) \cong H^*(\Fg, K;\Rb) \ot  S^*(\Fg^*)^G ,
\end{align}
where $K$ is the maximal compact subgroup. 
In the case when $G$ itself is compact one has
\begin{align*}
 H_{\rm cont}^0\bigl(G; S^*(\Fg^*)\bigr) \cong S^*(\Fg^*)^G \quad \text{and} \quad
 H_{\rm cont}^k\bigl(G; S^*(\Fg^*)\bigr), \quad \fl k >0 ,
\end{align*}
and so \eqref{BvE}and \eqref{vE} imply that
there is no cohomology above the diagonal. This allows Bott to conclude (\cf~\cite[Remark (a)]{Bott-ChWe}) that 
the edge homomorphism to the cohomology of the total complex
\begin{align*}
S^*(\Fg^*) \rightarrow H^*(\Om^*(NG)) ,
\end{align*} 
which under the identification 
$H^*(\Om^*(NG)) \cong H^*(BG)$ coincides with the Chern-Weil homomorphism,
is an isomorphism.

\

For the remainder of this section, we specialize to the group $\GL(n, \Rb)$ 
in order to remain consistent with the context of \S \ref{Hn}. 
To keep the notation convenient though
we will simply denote it by $G$ (except for the rare occasion when $G$ is allowed to 
be an arbitrary Lie group, in which case this will be explicitly stated). 

Returning to the Bott spectral sequence, where now for $G=\GL(n, \Rb)$ and 
$K={\rm O_n}$, it is well known that
\begin{align}  \label{E}
\begin{split}
H^*(\Fg, K; \Rb)  &\cong E (h_1, h_3, \ldots , h_{2 [\frac{n}{2}] +1}) , \\
 \text{and} \quad S^*(\Fg^*)^G & \equiv P[c_1, c_2, \ldots, c_n] ;
\end{split}
\end{align}
here $E$ stands for the exterior algebra in generators $h_i$ of degree $\deg h_i =2i-1$,
and $ P[c_1, c_2, \ldots, c_n]$ denotes
the polynomial algebra in generators $c_i$ of degree $\deg c_i =2i$.
Since the sequence converges to $H^*(BG) =
P[c_2, c_4..., c_{2 [\frac{n}{2}]}]$, it can be seen, successively, that $d_{2i-1}$ sends $h_{2i-1}$ to 
a non-zero multiple of $c_{2i-1}$, and that $c_{2i}$'s survive; also, $d_r =0$ for 
$r > 2 [\frac{n +1}{2}]$.

A similar pattern occurs for the spectral sequence of the truncated complex 
$(\Om^*(NG))_n, d, \d)$, the main
difference being that the polynomial algebra $P[c_1, c_2, \ldots, c_n]$ 
is replaced by its truncation modulo the ideal of elements of degree $> 2n$,
denoted $S^*(\Fg*)_n^{\GL_n} \equiv P_n[c_1, c_2, \ldots, c_n] $. 
In particular
\eqref{BvE} and \eqref{vE} become
\begin{align} \label{BvEn}
 H_\d^p \bigl(\Om^q(NG)_n\bigr) \cong H_{\rm cont}^{p-q}\bigl(G; S^q(\Fg^*)_n\bigr),
\end{align} 
respectively
\begin{align} \label{vEstn} 
 H_{\rm cont}^*\bigl(G; S^*(\Fg^*)_n\bigr) \cong H^*(\Fg, K;\Rb) \ot  S^*(\Fg^*)_n^G .
\end{align}
Thus, in view of \eqref{E}, the corresponding $E_1$-term is 
\begin{align} \label{wo} 
 E_1 \cong WO_n :=
 E (u_1, u_3, \ldots , u_{2 [\frac{n}{2}] +1}) \ot P_n[c_1, c_2, \ldots, c_n] .
\end{align}
There is a natural 
inclusion of the DG algebra $WO_n$ into the quotient  $W(\Fg, K)_n$ of  $W(\Fg, K)$ by 
 the ideal of elements of degree $> 2n$, and the map  $WO_n \hookrightarrow W(\Fg, K)_n$ 
is a quasi-isomorphism. Indeed, this follows from the comparison theorem for
the spectral sequences associated to the quotient of the standard
filtration (by the ideal generated by the elements of degree $>2n$), as the
$E_1$ term of both spectral sequences is the same as in \eqref{wo}.
 Although the spectral sequences of $W(\Fg, K)_n$ and $\Om^*(NG)_n$ cannot be directly
 compared, by chasing their differentials one can still infer that
 $H^*(W(\Fg, K)_n)$ and $H^*(\Om^*(NG)_n)$ are formally isomorphic. 
 
 The elegant argument though, due to Shulman and Stasheff~\cite{ShSt}, involves the 
semi-simplicial Weil algebra of Kamber and Tondeur~\cite{KT}. Relying on a semi-simplicial
 generalization of the van Est isomorphism, Shulman and Stasheff arrive
 to the isomorphism $H^*(W(\Fg, K)_n) \cong H^*(\Om^*(NG)_n)$ not by ad hoc
calculations but by exploiting the standard comparison theorem for 
spectral sequences. 
To explain their line of argument, we recall that the semi-simplicial Weil
algebra $W_1(\Fg)$ (see~\cite{KT, KT2}) is a $\Fg$-DG algebra with
underlying vector space $\bigoplus_{p \geq 0}W(\Fg^{p+1})$  
and faces induced by the projections ${\rm pr}_{k}: \Fg^{p+1} \to \Fg^p$, $0\leq k \leq p$,
that omit the $(k+1)$th factor.
The projection $\pi:W_1(\Fg) \to W(\Fg)$ on the first summand ($p=0$) is a map of $\Fg$-DG algebras, and is shown to induce a chain equivalence. This is but a special case of a core
result of Kamber and Tondeur (\cite[Theorem 8.12]{KT}, also \cite[\S 6]{KT2}), which for 
$W_1(\Fg)$ can be stated as follows. 

\begin{theorem}[\cite{KT}, Ch.8] Let $G$ be an almost connected Lie group and $H$ a
closed subgroup. There is a suitable filtration $\Fc_1$ of $W_1(\Fg)$
such that  the canonical projection at the level of truncated algebras
$$
\pi_m: W_1(\Fg, H)_m:=W_1(\Fg, H)/\Fc_1^{m+1}W_1(\Fg, H) \to W(\Fg, H)_m
$$ 
induces an isomorphism of the associated spectral sequence  
and hence on cohomology:
\begin{align} \label{KT}
\pi_m^\natural :  H^*(W_1(\Fg, H)_m) \xrightarrow{\cong} H^*(W(\Fg, H)_m) .
\end{align}
\end{theorem}

On the other hand, by the universal property of Weil algebras, the canonical connection 
on the principal
$G$-bundle $\bar{N}G \to NG$ gives rise to a map of simplicial $\Fg$-DG algebras
\begin{align} \label{varfi}
 \varphi: W_1(\Fg)= \bigoplus_{p \geq 0}W(\Fg^{p+1})  \to 
\bigoplus_{p \geq 0}\Om^*(G^{p+1})= \Om^*(\bar{N}G).
\end{align} 
Shulman and Stasheff~\cite{ShSt} define a compatible filtration $\Fc$ on $ \Om^*(\bar{N}G)$ which 
renders $\varphi$ fitration-preserving, as follows. 
First they endow  $\Om^*(NG)$ with the standard filtration $F$ by the degree 
of forms. Then they define $\Fc$ as the filtration
of $ \Om^*(\bar{NG})$ induced by $\psi:  \Om^*(NG) \to \bar{N}G)$, where $\psi=\{\psi_p\}$
with $\psi_p$ corresponding to the projection $G^{p+1}\to G^p$ on the first $p$ coordinates. 
In the
homogeneous picture $\Om^*({N}G)$ is identified with the subcomplex  $\Om^*(\bar{N}G)_G$ 
of $G$-basic forms on $\bar{N}G$. Note that
$\Fc$ also restricts to a filtration on the $K$-basic forms $\Om^*(\bar{N}G)_K$. 

Both $\varphi$ and $\psi$ descend to maps $\varphi_m$, resp. $\psi_m$, 
at the level of the quotient algebras by the $(m+1)$th power of the corresponding ideal,
which are shown to give rise to equivalences of spectral sequences.

\begin{theorem}[\cite{ShSt}, pp. 68-70] \label{semi-vE} 
Each of the maps  
\begin{align} \label{phim}
& \varphi_m: W_1(\Fg, K)/\Fc_1^{m+1} \to \Om^*(\bar{N}G)_K/\Fc^{m+1} , 
\quad \text{resp.} \\  \label{psim}
&\psi_m: \Om^*(NG)_m \equiv
\Om^*(\bar{N}G)_G /F^{m+1} \to \Om^*(\bar{N}G)_K/\Fc^{m+1} 
\end{align} 
is filtration preserving and induces an isomorphism of spectral sequences.  
 \end{theorem}
 
One thus obtains in homology the isomorphisms 
\begin{align} \label{WK}
\varphi_m^\natural : H^*\left(W_1(\Fg, K)_m \right) &\xrightarrow{\cong} 
H^*\left(\Om^*(\bar{N}/K)_m\right) , \quad \fl \, m \geq -1\\ \label{GtoK}
\psi_m^\natural : H^*\left(\Om^*(NG)_m\right) &\xrightarrow{\cong} 
H^*\left(\Om^*(\bar{N}/K)_m\right),
\end{align}
which combined with  \eqref{KT} give rise to a canonical isomorphism 
\begin{align} \label{nuHH} 
 H^*\left((\Om^*(NG)_m\right) & \cong H^*\bigl(W(\Fg, K)_m\bigr) .
\end{align}
The above isomorphism coincides with that of  \eqref{WH}, as both are
constructed via the universal property of the Weil algebra in Chern-Weil theory.
The difference is that the cross-section $s$ in \eqref{mu} is replaced by the equivalence
\eqref{psim} of spectral sequences, and the passage through the Dupont complex is 
altogether bypassed.

In order to see that Theorem \ref{semi-vE} represents
a semi-simplicial generalization of the van Est isomorphism 
$H^*(\Fg, K; \Rb) \cong H_{\rm cont}^*(G; \Rb)$, note that
at the level of $0$-th filtration,
$W_1(\Fg, K)/\Fc_1=  \bigoplus_{p \geq 0}\left(\wedge^* \Fg^{p+1}\right)_{K-{\rm basic}}$
while $\Om^0(NG) = \bigoplus_{p \geq 0}C^\infty(G^p)$ is the complex of 
differentiable group cohomology, and
$\varphi_0$ and $\psi_0$ are the very same maps used by van Est in his proof~\cite{vEst}.

 \
 
\subsection{Hopf algebra of algebraic forms} \label{salg} 
Regarding $G=\GL(n, \Rb)$ as a real algebraic group, we denote by 
$\Om^*_{\rm alg}(G)$ its graded Hopf algebra of algebraic differential forms. As vector
spaces $\Om^*_{\rm alg}(G) \cong C_{\rm alg}(G)~\ot~\bigwedge \Fg^*$,
where $C_{\rm alg}(G) = \Rb [g_{ij} ; \det g^{-1}] $ is the ring of regular rational functions 
on $G$, which is the classical Hopf algebra of a linear algebraic group. The associated 
Hochschild complex of the DG Hopf algebra $\Om^*_{\rm alg}(G)$ coincides with
the Bott simplicial de Rham $(\Om_{\rm alg}^{*,*}(NG), \d, d)$, where 
$\Om_{\rm alg}^{p,q}(NG):=\Om_{\rm alg}^q(G^p)$. 
  
\begin{thm} \label{alg}
There are canonical isomorphisms 
\begin{align}   \label{algHH}   
HH^q\left((\Om_{\rm alg}^*G)_n\right) &\xrightarrow{\cong} 
H^*\bigl(W(\Fg, G)_n\bigr) , \\ \label{algHC} 
HC^q\left((\Om_{\rm alg}^*G)_n\right)& \xrightarrow{\cong} 
 \bigoplus_{i \geq 0}H^{q-2i}\bigl(W(\Fg, G)_n\bigr).
\end{align}
\end{thm}
\begin{proof} 
The short proof is parallel to Bott's argument in the case of a compact Lie group
(\cf  ~\cite[Remark (a)]{Bott-ChWe}). Indeed, by the analogue of Bott's Decomposition Lemma,
the $E_1$-term of the spectral sequence for
the column filtration of the truncated complex $(\Om_{\rm alg}^{*,*}(NG)_n, \d, d)$ is
seen to be
\begin{align} \label{BoHo1}
E_1^{p,q} = H_\d^p \bigl(\Om_{\rm alg}^q(G^p)_n\bigr) \cong 
H_{\rm alg}^{p-q}\bigl(G; S^q(\Fg^*)_n\bigr) .
\end{align} 
The algebraic counterpart of the van Est isomorphism, which is a consequence of
Hochschild's isomorphism~\cite[Theorem 5.2]{Hoch}, then implies:  
\begin{align} \label{BoHo2}
\begin{split}
H_{\rm alg}^{0}\bigl(G; S^q(\Fg^*)_n\bigr) &\cong S^q(\Fg^*)_n^G, \quad \text{and} \\ 
 H_{\rm alg}^{i}\bigl(G; S^q(\Fg^*)_n\bigr)  &= 0 ,   \qquad \text{if} \quad i>0 .
 \end{split}
\end{align}
From \eqref{BoHo1} and \eqref{BoHo2} it follows that  
 \begin{align*}
 H^*(\Om_{\rm alg}^*(NG)_n\bigr)
\cong S^*(\Fg^*)_n=P_n [c_1, c_2, \ldots, c_n] .
 \end{align*} 
The left hand side is the same as $HH^*\left(\left(\Om_{\rm alg}^*G)_n\right)\right)$
 while the right hand is the same as $H^*\bigl(W(\Fg, G)_n\bigr)$.  
 
 The Shulman-Stasheff argument in \S \ref{diff} can also be adapted. The
 algebraic counterpart of $\psi_m$ in Theorem \ref{semi-vE} is $\psi_m^{\rm alg} = \Id$,
and the proof that  
 $$ 
 \varphi_m^{\rm alg}: W_1(\Fg, G)/\Fc_1^{m+1}  \to \Om_{\rm alg}^*(\bar{N}G)_G/\Fc^{m+1} \equiv
 \Om_{\rm alg}^*(NG)_m
 $$
 induces an isomorphism of spectral sequences follows along the same lines as in~\cite{ShSt},
 with the difference that the semi-simplicial generalization of the van Est bicomplex~\cite[\S 10]{vEst0}
 is replaced by its algebraic counterpart, which can be handled as in the proof of
  \cite[\S 2]{Kumar} for the algebraic version of the van Est isomorphism.
 
 We finally recall that \eqref{algHC} automatically follows from  \eqref{algHH}.
 \end{proof}
 
 \subsection{Hopf algebra of germs of forms} \label{sgerm}
We now consider a Hopf algebra of forms on the {\em group germ} determined by $G$, 
which is defined as follows.
First, a differential form $\a \in \Om^q(G)$ will be called {\em locally trivial} if it vanishes identically
in a neighborhood of $\1b :=\Id$ in $G$. We denote the set of such forms by $\Om_0^*(G)$. 
More generally, for any $p \in \Nb$, we define in a similar way the set $\Om_0^*(G^p)$ 
of locally trivial forms for the group $G^p$. Obviously, 
$\Om_0^*(G^p)$ is an ideal in $\Om^*(G^p)$, and we form the quotient algebra
$$
\bar{\Om}^*(G^p) :=  \Om^*(G^p)/\Om_0^*(G^p) , \quad  \fl p\in \Nb . 
$$
We then define the tensor product $\bot$ of two such algebras by stipulating
$$
\bar{\Om}^*(G^p) \bot \bar{\Om}^*(G^q) := \bar{\Om}^*(G^{p+q}),
$$
and equip $\bar{\Om}^*(G)$ with the coproduct 
$$
\bar\D: \bar{\Om}^*(G) \to \bar{\Om}^*(G) \bot \bar{\Om}^*(G)
$$
induced by $\D: \Om^*(G) \to \Om^*(G) \ot \Om^*(G)$. The latter is well-defined since
if $\a \in \Om_0^*(G)$ then $\D\a \in \Om_0^*(G^2)$.  
With this coproduct and the obvious antipode, unit and counit, 
it is straightforward to verify that $\bar{\Om}^*(G)$ satisfies the axioms of
a DG Hopf algebra. Its Hochschild complex coincides with 
the simplicial de Rham bicomplex $\{\bar{\Om}^{*,*}(NG), \d, d\}$, 
resp. $\{\bar{\Om}^{*,*}(\bar{N}G), \d, d\}$, where 
$\bar{\Om}^{p,q}(NG):=\bar{\Om}^q(G^p)$ and 
$ \bar{\Om}^{p,q}(\bar{N}G):=\bar{\Om}^q(G^{p+1})$.

\begin{thm} \label{germ}
There are canonical isomorphisms 
\begin{align}   \label{germHH}     
 HH^q\left((\bar{\Om}^*G)_n\right)&\xrightarrow{\cong} 
 H^*\bigl(W(\Fg)_n\bigr) , \\ \label{germHC} 
HC^q\left((\bar{\Om}^*G)_n\right)& \xrightarrow{\cong} 
 \bigoplus_{i \geq 0}H^{q-2i}\bigl(W(\Fg)_n\bigr).
\end{align}
\end{thm}

\begin{proof} 
The analogue of Bott's Decomposition Lemma yields
 \begin{align*}
E^{p,q}_1 =  H_\d^p \left(\bar{\Om}^q (G^p)_n\right) \cong 
 H_{\square}^{p-q} \left(G; S^q(\Fg^*)_n\right) ,
 \end{align*}
where $H_\square^*$ refers to the cohomology of group germs,
which by its very definition  (see~\cite[\S 4]{Sw}) in the case at hand is the cohomology of
the complex $\{\bar{\Om}^0 \left(G^p ; S^q( \Fg^*)_n\right),  \d \}$. 
Furthermore,  the counterpart of van Est's isomorphism
for this case was proved by \'Swierczkowski~\cite[Theorem 2]{Sw} and it gives
the isomomorphism
 \begin{align} \label{sw}
H_{\square}^{i} \left(G; S^q( \Fg^*)_n\right) \cong 
 H^{i} \left(\Fg; S^q( \Fg^*)_n\right). 
 \end{align}
 Since $\, H^{\ast} \left(\Fg; S^q(\Fg^*)_n\right) \cong
 H^{\ast} (\Fg; \Rb) \ot S^q( \Fg^*)_n^{G}$, it follows that 
  \begin{align*}
 H^{\ast} (\Fg; \Rb) \ot S( \Fg^*)_n^{G}  
 \cong E [h_1, h_2, \ldots , h_n] \ot P_n [c_1, c_2, \ldots, c_n]
 \end{align*}  
After identifying the first differential as being
 $d_1 u_i=c_i$ for $1\leq i \leq n$, \ie showing that $E_1\cong W_n$, 
 one can continues as in the first proof in \S \ref{diff}. 
 
 Again, the more satisfying argument follows the pattern of the second proof in \S \ref{diff}, 
 with the maps
 $\varphi_m$ and $\psi_m$ replaced by \, 
  $ \varphi_m^{\rm germ}: W_1(\Fg)/\Fc_1^{m+1}  \to \bar{\Om}^*(\bar{N}G)/\Fc^{m+1}$
 and 
 $ \psi_m^{\rm germ}: \bar{\Om}^*(NG)_m \to \bar{\Om}^*(\bar{N}G)/\Fc^{m+1}$.
 The proof that these maps induce isomorphisms of spectral sequences goes
 along the same lines as in
 \cite{ShSt}, with the necessary modifications similar to those used by \'Swierczkowski
to adapt van Est's proof for the bicomplex in~\cite[\S 10]{vEst0} to the bicomplex 
in~\cite[\S 7]{Sw}.
 \end{proof}
 
\subsection{Examples of representative cocycles} \label{fin}
 First we point out that, similarly to the realization of $\Hc_n= \Hc(n, \Rb)$ 
 as a Hopf algebra of `moving frames', $\Om^*((\GL(n, \Rb))$ has a natural representation
 as a Hopf algebra of `moving coframes'. This representation is given by the coaction of 
$G =\GL(n, \Rb)$  on the \'etale groupoid $\G_n= \Diff_n \ltimes \Rb^n$ determined by 
the first jet map (analogous to the local coaction of $\GL(n, \Rb)$ on the $C^*$-algebra 
of a codimension $n$ foliation, mentioned by Connes in~\cite[III 7 $\a$]{book}). Specifically,
 \cf \cite[\S 5]{Gor}, the formula 
 \begin{align} \label{rho}
\rho(\a) (\om U_\vp) := J^*(\a)\om U_\vp , \quad \a \in \Om^*(G), \quad \om \in \Om_c^*(\G_n) ,
\end{align}
where $J:\G_n \to \GL(n, \Rb)$ is the Jacobian map 
\begin{align*}
J(\vp, x) =  \vp^{\prime}(x) = \left(\part_{j} \vp^{i}  (x)\right) \in  \GL(n, \Rb) ,
\end{align*}
defines an action  $\rho : \Om^*(G) \times \Om_c^*(\G_n) \to  \Om_c^*(\G_n) $.

Parallel to the case of the action of $\Hc_n$ on $\G_n$, 
there is a closed graded trace $\cutint: \Om_c^*(\G_n) \to \Cb$, given by
\begin{align*}  
\cutint \om U_\vp = \int \om ,  \quad  \text{if} \,\, \vp=\1b , \qquad 
 \cutint f U_\vp = 0 , \quad  \text{if} \,\, \vp \neq \1b ,
\end{align*}
which is $\Om^*(G)$-invariant, \ie satisfies
\begin{align*}  
\cutint (d\om) U_\vp = 0 ,  \quad  \text{and} \qquad 
 \cutint \rho(\a) (\om U_\vp)= \epsilon(\a)\cutint \om U_\vp .
\end{align*}
With these at hand, one can define (see~\cite[(3.11)]{Gor})
a characteristic map of cyclic modules
$\kappa_*: \Om^*(G)^\natural \to \Om_c^*(\G_n)^\natural$ by setting
\begin{align}  \label{kappa}
 \kappa_q (\a^1 \ot \ldots \ot \a^{q})(a_0, a_1, \ldots, a_q) := 
(-1)^{\rm d} \cutint a_0 \rho(\a^1)(a_1) \cdots \rho(\a^q)(a_q) , 
 \end{align}
where $\a^i \in \Om^*(G)$, $a_j \in  \Om_c^*(\G_n)$, and  ${\rm d}:=\sum_{i >j} \deg \a^i \deg a_j$.   

In order to write down examples of Hopf cyclic classes in $HC^*((\Om^*(G))_n)$, 
let $\theta \in \Om^*(G)\ot \Fg^*$ denote the Maurer-Cartan form 
$\theta_g = g^{-1}dg$, and define $\log|\det| \in \Om^0(G)$ by
$\log|\det| (g):= \log |\det g|$. Then
\begin{align*}  
gv_{1, q} := \log|\det| \ot \Tr \left(\theta^{\ot q}\right) \in \Om^q(G^{q+1})
\end{align*}  
is a cocycle representing in $HC^*((\Om^*(G))_n)$ the Godbillon-Vey class corresponding
to $u_1c_1^q \in H^{2q+1}(W(\Fg, G)_n)$. More generally, cocycles representing
generalized Godbillon-Vey classes in $HC^*((\Om^*(G))_n)$ can be obtained as suitable linear
combinations of terms of the form
\begin{align*}  
 \log|\det| \ot \Tr \left(\theta^{\ot \ell_1}\right) \ot \Tr \left(\theta^{\ot \ell_2}\right) \ot \cdots
\ot \Tr \left(\theta^{\ot \ell_k}\right) \in \Om^q(G^{q+1})
\end{align*}
where $\lb = (\ell_1, \ldots, \ell_k)$ runs over the partitions of $q$. Carried through
the characteristic map \eqref{kappa}, these cocycles acquire expressions reminiscent of
those obtained for the same classes by Bott~\cite{Bott-gv} in group cohomology, and by
Crainic and Moerdijk~\cite[\S 5.1]{CraiM} in their \v{C}ech-de Rham theory for leaf spaces. 

The Chern classes in both
$HC^*((\Om^*(G))_n)$ and $HC^*((\Om_{\rm alg}^*(G))_n)$ can also be represented by
cocycles of a similar form only without the transcendental factor $\log|\det|$.

 
  \end{document}